\documentclass[12pt,reqno]{amsart}
\usepackage{amsmath, amssymb, latexsym, amscd, amsthm,amsfonts,amstext}
\usepackage[top=3.5cm, bottom=3.5cm, left=3.0cm, right=2.5cm]{geometry} 
\usepackage[mathscr]{eucal}
\usepackage{xspace}
\usepackage{array, tabularx, longtable}
\usepackage{multicol,color}
\usepackage{indentfirst}
\usepackage{graphics}
\usepackage[colorlinks=true]{hyperref}
\theoremstyle{plain}
\newtheorem{theorem}{Theorem}
\newtheorem{corollary}[theorem]{Corollary}
\newtheorem{main}{Main theorem}
\newtheorem{lemma}[theorem]{Lemma}

\theoremstyle{remark}

\theoremstyle{definition}

\begin{document}
\date{} 
\title [Weak Dirac conjecture]{A new progress on Weak Dirac conjecture}

\author{Hoang Ha Pham and Tien Cuong Phi}
\address{Department of Mathematics, Hanoi National University of Education, 136 XuanThuy str., Hanoi, Vietnam}
\email{ha.ph@hnue.edu.vn; cuong.tienphi@gmail.com }
\begin{abstract}
	In 2014, Payne-Wood proved that every non-collinear set $P$ of $n$ points in the Euclidean plane contains a point in at least $\dfrac{n}{37}$ lines determined by $P.$ This is a remarkable answer for the conjecture, which was proposed by Erd\H{o}s, that every non-collinear set $P$ of $n$ points contains a point in at least $\dfrac{n}{c_1}$ lines determined by $P$, for some constant $c_1.$ In this article, we refine the result of Payne-Wood to give that every non-collinear set $P$ of $n$ points contains a point in at least $\dfrac{n}{26}+2$ lines determined by $P$ . Moreover, we also discuss some relations on theorem Beck that every set $P$ of $n$ points with at most $l$ collinear determines at least $\dfrac{1}{61}n(n-l)$ lines.
\end{abstract}
\maketitle
\section{Introduction}
Let $P$ be a set of points in the Euclidean plane. A line that contains at least two points in $P$ is said to be determined by $P.$

In 1951, G. Dirac (\cite{D}) made the following conjecture, which
remains unsolved:\\
{\bf Conjecture 1 (Strong Dirac Conjecture).} Every non-collinear set $P$ of $n$ points in the plane contains a point in at least $\dfrac{n}{2}-c_0$ of the lines determined by $P$, for some constant $c_0$.

In 2011, J. Akiyama, H. Ito, M. Kobayashi, and G. Nakamura (\cite{AIKN}) gave somes examples to show that the $\dfrac{n}{2}$ bound would be tight. We note that if $P$ is non-collinear and contains $\dfrac{n}{2}$ or more collinear points, then Dirac's Conjecture holds.Thus we may assume that $P$ does not contain $\dfrac{n}{2}$ collinear points, and $n \geq 5.$

In 1961, P. Erd\H{o}s (\cite{E}) proposed the following weakened conjecture.\\
{\bf Conjecture 2 (Weak Dirac Conjecture).} Every non-collinear set $P$ of $n$ points contains
a point in at least $\dfrac{n}{c_1}$ lines determined by $P,$ for some constant $c_1.$\\

In 1983, Beck (\cite{B}) and Szemer\'{e}di-Trotter (\cite{ST}) proved the Weak Dirac Conjecture for the case $c_1$ but it is unspecied or very large. In 2014, Payne-Wood (\cite{PW}) proved the following theorem:
\begin{theorem} Every non-collinear set $P$ of $n$ points contains a point in at least $\dfrac{n}{37}$ lines determined by $P$.
	\end{theorem}

For the first purpose of this article, we would like to give a new progress for the Weak Dirac conjecture. In particular, we prove the following:
\begin{main}\label{m1} Every non-collinear set $P$ of $n$ points contains a point in at least $\dfrac{n}{26} + 2$ lines determined by $P.$
	\end{main}

Moreover, relate to work on the Weak Dirac Conjecture, Beck gave the number of lines determined by $P.$ He proved the following theorem.
\begin{theorem}(\cite{B}) Every set $P$ of $n$ points with at most $l$ collinear determines at least $c_2n(n-l)$ lines, for some constant $c_2.$
	\end{theorem}
In 2014, Payne - Wood also gave a remarkable improvement of Beck's theorem by proving the following.
\begin{theorem}(\cite{PW}) Every set $P$ of $n$ points with at most $l$ collinear determines at least $\dfrac{1}{98}n(n-l)$ lines.
	\end{theorem}
We note that the number 98 can be instead by 93. The details can be found in \cite{P}.

For the final purpose, we would like to give some results for the number of lines with few points from $n$ points in plane. Then, we also give the following theorems.
\begin{main} \label{m2} Every set $P$ of $n$ points with at most $l$ collinear determines at least $\dfrac{1}{61}n(n-l)$ lines.
	\end{main}
\begin{main}\label{m3}
	Every set $P$ of $n$ points with at most $l$ collinear determines at least $\dfrac{1}{122}n(n-l)$ lines with at most 3 points.
\end{main}
\section{Auxiliary Results}
We list here some known results which are very helpful for the proofs of the main theorems.\\
The crossing number of a graph $G$, denoted by $cr(G),$ is the minimum number of
crossings in a drawing of $G.$ The following version due to J. Pach, R. Radoi\v{c}i\'{c}, G. Tardos and G. T\'{o}th \cite{PR} is the
strongest to date.
\begin{lemma} (Crossing lemma \cite{PR}).  \label{CL}
	For every graph with $n$ vertices and $m\geq \dfrac{103}{16}n$ edges, then
	\begin{equation*}
	cr(G)\geq \dfrac{1024m^3}{31827n^2}.
	\end{equation*}
\end{lemma}
We set $E(H)$ to be the set of all edges of a graph $H.$ The visibility graph $G$ of a point set $P$ has vertex set $P$ , where $vw \in E(G)$ whenever the line segment $vw$ contains no other point in P (that is, $v$ and $w$ are consecutive on a line determined by $P$ ). For $i \geq 2,$ an $i-line$ is a line containing exactly $i$ points in $P.$ Let $s_i$
be the number of $i-lines.$ Let $G_i$ be the spanning subgraph of the visibility graph of $P$ consisting of all edges in $j-lines$ where $j\geq i.$ Note that since each $i-line$ contributes $i-1$ edges,$|E(G_i)| = \sum_{j\geq i}(j-1)s_j.$ We introduce some useful results:
\begin{theorem}\label{H} (Hirzebruch's Inequality \cite{H}). Let $P$ be a set of $n$ points with at most $n-3$ 	collinear. Then 
	\begin{equation*}
	s_2+\dfrac{3}{4}s_3 \ge n+ \sum\limits_{i \ge 5} (2i-9)s_i.
	\end{equation*}
\end{theorem} 
\begin{theorem}\label{ST} (Szemer\'{e}di-Trotter \cite{ST}). Let $\alpha$ and $\beta$ be positive constants such that every graph $H$ with $n$ vertices and $m \geq \alpha n$ edges satisfies
\begin{equation*}
cr(H) \geq \dfrac{m^3}{\beta n^2}
\end{equation*}	
	Let $P$ be a set of $n$ points in the plane. Then
\begin{align*}
& a)\quad \left|E(G_i)\right|=\sum\limits_{j \ge i}(j-1)s_j \le max\{\alpha n,\dfrac{\beta n^2}{2(i-1)^2}\},\\
& b)\quad \sum_{j \ge i}s_j \le max\{\dfrac{\alpha n}{i-1}, \dfrac{\beta n^2}{2(i-1)^3}\}.
\end{align*}
\end{theorem}
\section{A new progress on Weak Dirac's conjecture}
In order to get the main theorem \ref{m1}, we refine the method of Payne-Wood to find the largest number $\varepsilon$ such that every set $P$ of $n$ non-collinear points in the plane at most $\varepsilon n +2$ collinear points, the arrangement of $P$ has at least $\varepsilon n^2+2n$ point- line incidents. We start by the following result.
\begin{theorem}\label{Theorem1} Let $\alpha$ and $\beta$ be positive constants such that every graph $G$ with $n$ vertices and $m \ge \alpha n$ edges satisfies $cr(G) \ge \dfrac{m^3}{\beta n^2}$.\\
	
	Fix two integers $c \ge 8, 0\leq q\leq 3$ and a real number $\epsilon \in (0;\dfrac{1}{2}), \epsilon n \geq 2$. Let $h:= \dfrac{c(c-2)}{5c-18}$. Then for every set $P$ of $n$ points in the plane with at most $\epsilon n + q$ collinear points, the arrangemnet of $P$ has at least $\delta n^2+rn$ point- line incident,\\
	where 
	\begin{align*}
	&\delta = \dfrac{1}{h+1} \left(1- \epsilon \alpha -\dfrac{\beta}{2} \left( \dfrac{-18(c-2)}{c^3(5c-18)} +\sum \limits_{i \ge c}\dfrac{i+1}{i^3} \right) \right),\\ &r=\dfrac{2h-1+\alpha}{h+1}.
	\end{align*}
\end{theorem}
\begin{proof}
	Let $J=\{2;3;...;\left\lfloor {\epsilon n} \right\rfloor +q \}$ and assume that $\epsilon n \geq 2.$ Considering the visibility graph $G$ of $P$ and its subgraphs $G_i$ as defined previously. Let $k$ be the minimum integer such that $|E(G_k)| \le \alpha n$. If there is no such $k$ then let $k= \left\lfloor {\epsilon n} \right\rfloor+q +1$. An integer $i \in J$ is \textit{large} if $i \ge k$ , and is \textit{small} if $i \le c.$ An integer in $J$ that is neither small nor large is \textit{medium}.
	\vskip 0.2cm
	\noindent Recall that an $i$-line is a line containing exactly $i$ points in $P$.
	An \textit{$i$-pair} is a pair of points in an $i$-line. A \textit{small pair} is an $i$-pair for some small $i$. Define   \textit{large pair}, \textit{medium pair} analogously. Let $P_S, P_M$ and $P_L$ denote the number of small, medium and large pairs respectively. An \textit{i-incidence} is an incidence between a point of $P$ and an $i$-line. A \textit{small incidence} is an $i$-incidence for some small $i$, and define \textit{medium, large incidences} analogously. Let $I_S, I_M$ and $I_L$ denote the number of small, medium and large incidences respectively and let $I$ denote the total number of incidences. Since every $s_i$ has $i$ points incidence with its, then  
	\begin{equation*}
	I=\sum \limits_{i \in J} is_i=I_S+I_M+I_L
	\end{equation*}
	Because  $P$ has no more than $\dfrac{n}{2}$ collinear points and $n \ge 5$, thus $ \lfloor \epsilon n \rfloor +q \leq  \lfloor \dfrac{n}{2} \rfloor \leq n-3$. Therefore, for $n$ points of $P$ has no more than $n-3$ collinear points. Applying the Hirzebruch's Inequality (Theorem \ref{H}), we have 
	\begin{equation*}
	s_2+\dfrac{3}{4}s_3 \ge n+ \sum\limits_{i \ge 5}(2i-9)s_i.
	\end{equation*}
	Since $h>0$ then, 
	$$hs_2+\dfrac{3}{4}hs_3-hn-h\sum\limits_{i \ge 5}(2i-9)s_i \ge 0 .$$
	\begin{eqnarray*}
		P_S &=& \sum \limits_{i=2}^{c} \left(\begin{array}{*{20}{c}}
			i\\2
		\end{array} \right)
		s_i \\
		&=& s_2+ 3s_3+ 6s_4 +\sum\limits_{i=5}^{c}\left(\begin{array}{*{20}{c}}
			i\\2
		\end{array} \right)
		s_i\\
		&\le& (h+1)s_2+(\dfrac{3h}{4}+3)s_3+6s_4+\sum\limits_{i=5}^{c}\left(\begin{array}{*{20}{c}}
			i\\2
		\end{array} \right)
		s_i-hn-h\sum\limits_{i \ge 5}(2i-9)s_i\\
		&=& \dfrac{h+1}{2}.2s_2+\dfrac{h+4}{4}.3s_3+\dfrac{3}{2}.4h_4+\sum\limits_{i=5}^{c}\left(\dfrac{i-1}{2}-2h+\dfrac{9h}{i}\right)is_i\\
		&-&h\sum\limits_{i = c+1}^{k-1}(2i-9)s_i-h\sum\limits_{i \ge k}(2i-9)s_i-hn\\
		&\le& \dfrac{h+1}{2}.2s_2+\dfrac{h+4}{4}.3s_3+\dfrac{3}{2}.4h_4+\sum\limits_{i=5}^{c}\left(\dfrac{i-1}{2}-2h+\dfrac{9h}{i}\right)is_i\\
		&-&h\sum\limits_{i= c+1}^{k-1}(2-\dfrac{9}{c+1})is_i-h\sum\limits_{i \ge k}(2-\dfrac{7}{c})(i-1)s_i-hn.
	\end{eqnarray*}
	Setting $X:=max\{\dfrac{h+1}{2};\dfrac{h+4}{4};\dfrac{3}{2};max_{5\le i\le c}(\dfrac{i-1}{2}-2h+\dfrac{9h}{i})\}$ implies that,
	\begin{equation} \label{1}
	P_S \le X I_S-h\sum\limits_{i=c+1}^{k-1}(2-\dfrac{9}{c+1})is_i-h\sum\limits_{i \ge k}(2-\dfrac{7}{c})(i-1)s_i-hn.
	\end{equation} 
	Let $\gamma(h,i)=\dfrac{i-1}{2}-2h+\dfrac{9h}{i}$ \hspace{0.2 cm} for \hspace{0.2 cm} $5 \le i \le c$.\\
	We have: $\gamma_i^{"} \ge 0 \ \ \forall i \in (5,c)$ $\Rightarrow \gamma(h,i)_{max}= \gamma(h,5)=2- \dfrac{h}{5}$\\
	 or  $\gamma(h,i)_{max}= \gamma(h,c)=\dfrac{c-1}{2}-2h+\dfrac{9h}{c}$ for $c \ge 8$.
	
	Clearly, $h(c) =\dfrac{c(c-2)}{5c-18}$ has minimum value $\dfrac{24}{11}$ when $c=8$. Hence,
	\begin{eqnarray*}
		\dfrac{h+1}{2} &\ge& \dfrac{3}{2} \\
		\dfrac{h+1}{2} &\ge& \dfrac{h+4}{4} \\
		\dfrac{h+1}{2} &\ge& 2-\dfrac{h}{5} \\
		\dfrac{h+1}{2} &=& \dfrac{c-1}{2}-2h+\dfrac{9h}{c}.
	\end{eqnarray*}
	Thus, $X= \dfrac{h+1}{2}.$
	\vskip 0.2cm
	
	On the other hand, if $i \in J$ is medium ($c<i<k$) then $i$ is not large. Therefore, $\sum\limits_{j \ge i}(j-1)s_j >\alpha n$.
	Because if $\sum\limits_{j \ge i}(j-1)s_j \le \alpha n$ then $|E(G_i)| \le \alpha n$, contradict with minimum property of $k$.
	Using part (a) and (b) of the Szemerédi- Trotter theorem \ref{ST}, 
	\begin{equation}\label{2}
	\sum\limits_{j \ge i}js_j =\sum\limits_{j \ge i}(j-1)s_j+ \sum\limits_{j \ge i}s_j 
	\le \dfrac{\beta n^2}{2(i-1)^2}+\dfrac{\beta n^2}{2(i-1)^3}=\dfrac{\beta n^2 i}{2(i-1)^3}. 
	\end{equation}
	Given $X$ as above, we have
	\begin{eqnarray*}
		P_M-XI_M&=&\left(\sum\limits_{i=c+1}^{k-1}\left(\begin{array}{*{20}{c}}
			i\\2
		\end{array} \right)
		s_i\right)-X\left(\sum\limits_{i=c+1}^{k-1}is_i\right) \\
		&=&\dfrac{1}{2} \sum\limits_{i=c+1}^{k-1}(i-1-2X)is_i.
	\end{eqnarray*}
	Combining with (\ref{1}), we get
	\begin{align} \label{3}
	P_S+P_M \le XI_S-hn+XI_M+\dfrac{1}{2}\sum\limits_{i=c+1}^{k-1}\left(i-1-2X-4h+\dfrac{18h}{c+1}\right)is_i-h(2-\dfrac{7}{c})|E(G_k)|.
	\end{align} 
	We define 
	\begin{align*}
	Y&=c-5h-2+\dfrac{18h}{c+1}\\
	&=c -2- 5\dfrac{c(c-2)}{5c-18}+\dfrac{18c(c-2)}{(c+1)(5c-18)}\\
	&=\dfrac{-18(c-2)}{(c+1)(5c-18)}.
	\end{align*}
	This implies $-1 < Y < 0$ with $c\geq 8.$ Thus we have,
	\begin{align*}
		T= &\dfrac{1}{2}\sum\limits_{i=c+1}^{k-1}\left(i-1-2X-4h+\dfrac{18h}{c+1}\right)is_i= \dfrac{1}{2} \sum\limits_{i=c+1}^{k-1}(i-c+Y)is_i \\
		&=\dfrac{1}{2}\left(\sum\limits_{i=c+1}^{k-1}\sum\limits_{j= i}^{k-1}js_j\right)+\dfrac{Y}{2}\left(\sum\limits_{i= c+1}^{k-1}is_i\right)\\
		&\leq\dfrac{1}{2}\left(\sum\limits_{i=c+1}^{k-1}\sum\limits_{j\geq i}js_j\right)+\dfrac{Y}{2}\left(\sum\limits_{i\geq c+1}is_i\right).
	\end{align*}
		Applying (\ref{2}) and $Y+1 > 0$, this yields
	\begin{align}\label{4}
		T \le \dfrac{1}{2} \sum\limits_{i \ge c+1} \dfrac{\beta n^2i}{2(i-1)^3} +\dfrac{Y}{2}.\dfrac{\beta n^2(c+1)}{2c^3}= \dfrac{\beta n^2}{4}\left(Y\dfrac{c+1}{c^3}+\sum\limits_{i \ge c}\dfrac{i+1}{i^3}\right). 
		 	\end{align}
	Finally, we have
	\begin{equation} \label{5}
	P_L -XI_L=\sum\limits_{i=k}^{\left\lfloor \varepsilon n \right\rfloor + q}\left(\begin{array}{*{20}{c}}
	i\\2
	\end{array} \right)
	s_i-X\sum\limits_{i\geq k}is_i \le \dfrac{\epsilon n+q}{2} \sum\limits_{i \ge k}(i-1)s_i-X\sum\limits_{i\geq k}(i-1)s_i\\
	=(\dfrac{\epsilon n+q}{2}-X)|E(G_k)|.
	\end{equation}
	Combining (\ref{3}), (\ref{4}), (\ref{5}), we get
	\begin{align*}
	P_S+P_M+P_L& \leq X(I_S+I_M+I_L)-hn\\
	&+\dfrac{\beta n^2}{4}(Y\dfrac{c+1}{c^3}+\sum\limits_{i\ge c}\dfrac{i+1}{i^3})+\dfrac{1}{2}(\epsilon n+q-2X-4h+\dfrac{7h}{c})|E(G_k)|\\
	&\leq XI-hn+\dfrac{\beta n^2}{4}(Y\dfrac{c+1}{c^3}+\sum\limits_{i\ge c}\dfrac{i+1}{i^3})+\dfrac{1}{2}(\epsilon n-2)|E(G_k)|\ (\text{ by}\ 1\leq q\leq 3, c \geq 8)\\
	&\leq XI-hn+\dfrac{\beta n^2}{4}(Y\dfrac{c+1}{c^3}+\sum\limits_{i\ge c}\dfrac{i+1}{i^3})+\dfrac{1}{2}(\epsilon n-2)\alpha n.
	\end{align*}
	On the other hand, we have $P_S+P_M+P_L=	\left(\begin{array}{*{20}{c}}
	n\\2
	\end{array} \right)
	= \dfrac{1}{2}(n^2-n).$\\	
	Thus, we get
	\begin{eqnarray*}
		\dfrac{1}{2}(n^2-n)\le XI-hn+\dfrac{\beta n^2}{4}(Y\dfrac{c+1}{c^3}+\sum\limits_{i\ge c}\dfrac{i+1}{i^3})+\dfrac{\epsilon \alpha n^2}{2}-\alpha n.
	\end{eqnarray*}
	$\Rightarrow I \ge \dfrac{1}{2X}\left(1- \epsilon \alpha -\dfrac{\beta}{2}\left(Y\dfrac{c+1}{c^3}+\sum\limits_{i \ge c}\dfrac{i+1}{i^3}\right)\right)n^2+\dfrac{2h-1+\alpha}{2X}n$. \\
	
	Since $X=\dfrac{h+1}{2}$ and $Y=\dfrac{-18(c-2)}{c^3(5c-18)}$ then,
	\begin{align*}
	I &\ge \dfrac{1}{h+1} \left(1- \epsilon \alpha -\dfrac{\beta}{2} \left( \dfrac{-18(c-2)}{c^3(5c-18)} +\sum \limits_{i \ge c}\dfrac{i+1}{i^3} \right) \right)n^2 +\dfrac{2h-1+\alpha}{h+1}n\\
	&=\delta n^2 +rn.
	\end{align*}
\end{proof}
\begin{theorem}\label{t11}
	Every set $P$ of $n$ non-collinear points in the plane with at most $\dfrac{n}{26}+2$ collinear points, the arrangement of $P$ has at least $\dfrac{n^2}{26}+2n$ point- line incidents.
\end{theorem}
\begin{proof}
Case 1. If $ 0 < \dfrac{n}{26} < 1,$ then the arrangement of $P$ is $n^2-n > \dfrac{n^2}{26}+2n$ by $n \geq 5.$\\
Case 2. If $ 1\leq  \dfrac{n}{26} < 2,$ then $I = 2s_2+3s_3 \geq s_2+3s_3 = \dfrac{n^2-n}{2} > \dfrac{n^2}{26}+2n$ by $n \geq 26.$\\
Case 3. If $  \dfrac{n}{26} \geq 2,$ then the assumptions of Theorem \ref{Theorem1} satisfy with $\epsilon = \dfrac{1}{26}, c= 46, q=2.$ We have 
\begin{equation*}
I \geq \delta n^2 + rn \geq \dfrac{n^2}{26}+2n.
\end{equation*}
The proof of Theorem \ref{t11} is completed.
\end{proof}
So we now can give the proof of Main theorem \ref{m1}.
\begin{proof} Let $P$ be a set of $n$ non-collinear points in the plane. If $P$ contains at least $\dfrac{n}{26}+2$ collinear points, then every other point is in at least $\dfrac{n}{26}+2$ lines  $P$ ( one through each of the collinear points). Otherwise, by Theorem \ref{t11}, the arrangement of $P$ has at least $\dfrac{n^2}{26}+2n$ incidences, and so some point is incident with at least $\dfrac{n}{26}+2$ lines determined by $P$. Main theorem \ref{m1} is proved.
\end{proof}
We note that the number $\varepsilon = \dfrac{1}{26}$ is best possible in this technic. Indeed, for our purpose, we need $\delta \ge \varepsilon$ to get a constant $\varepsilon$ in Theorem \ref{Theorem1}. 
Using equivalent transformation, 
\begin{equation*}
\varepsilon \le  \dfrac{1- \dfrac{\beta}{2} \left( \dfrac{-18(c-2)}{c^3(5c-18)} +\sum \limits_{i \ge c}\dfrac{i+1}{i^3}\right)}{h+1+\alpha}= \dfrac{1- \dfrac{\beta}{2} \left( \dfrac{-18(c-2)}{c^3(5c-18)} +\sum \limits_{i \ge c}\dfrac{i+1}{i^3}\right)}{\dfrac{c(c-2)}{5c-18}+1+\alpha}.
\end{equation*}

In order to having maximum value $\varepsilon$ we need to optimal value $c$. We define
\begin{align*}
f(c)&= \dfrac{1- \dfrac{\beta}{2} \left( \dfrac{-18(c-2)}{c^3(5c-18)} +\sum \limits_{i \ge c}\dfrac{i+1}{i^3}\right)}{\dfrac{c(c-2)}{5c-18}+1+\alpha}
\end{align*} , for defined constant $\alpha,\beta$ in Crossing lemma \ref{CL}.
Using Maple application we have that the maximum value of $f(c)$ is at $c=46$. Hence, we can choose $\varepsilon >\dfrac{1}{26}.$ This shows that $\dfrac{1}{26}$ is the best constant.

\section{The lines with few points}
\begin{theorem} \label{Theorem2}	Let $\alpha, \beta$ be positive constants such that every graph $H$ with $n$ vertices and $m \ge \alpha n$ edges satisfies
	\begin{equation*}
	cr(H) \ge \dfrac{m^3}{\beta n^2}.
	\end{equation*} 
	Fix an integer $c \ge 29.$ Then for every set $P$ of $n$ points in the plane with at most $l$ collinear points, the arrangement of $P$ has at least
	\begin{align*}
	 \bigg(1-\dfrac{\beta}{2}\left(\dfrac{c^2-3c-14}{2c^3(c-4)}+\sum\limits_{i\geq c}\dfrac{1}{i^2}\right)\bigg)\dfrac{2c-8}{c^2+3c-18}n^2-\dfrac{(2c-8)\alpha}{c^2+3c-18}ln. 
	\end{align*}
		lines with at most $c$ points.
\end{theorem}
\begin{proof}
	We define the small, medium and large pairs and lines respectively as in the proof of Theorem \ref{Theorem1}. 	Set $h=\dfrac{c^2-c-2}{4c-16}$, where $c\ge 29$. Thus, $h>0$. Using the Inequality of Hirzebruch (Theorem \ref{H}), we have
	\begin{equation*}
	s_2+\dfrac{3h}{4}s_3 - hn - h\sum\limits_{i\geq 5}(2i-9)s_i \ge 0.
	\end{equation*}
	Now we have,
	\begin{eqnarray*}
		P_S &=& \sum\limits_{2}^{c}\left(\begin{array}{*{20}{c}}
			i\\2
		\end{array} \right)s_i \\  
		&=& s_2+3s_3+6s_4+\sum\limits_{i=5}^{c} \left(\begin{array}{*{20}{c}}
			i\\2
		\end{array} \right)
		s_i \\
		&\le& (h+1)s_2+\left(\dfrac{3h}{4}+3\right)s_3+6s_4+\sum\limits_{i=5}^{c}\left(\begin{array}{*{20}{c}}
			i\\2
		\end{array} \right)
		s_i-hn-h\sum\limits_{i\geq 5}(2i-9)s_i \\
		&\le& (h+1)s_2+\dfrac{3}{4}(h+4)s_3+6s_4+\sum\limits_{i=5}^{c}\left(\dfrac{i(i-1)}{2}-h(2i-9)\right)s_i-hn-h\sum\limits_{i\geq c+1}(2i-9)s_i. 
	\end{eqnarray*}
	By $c \geq 29,$ it is easy to see that
	$$X:=h+1=max\{h+1; \dfrac{3}{4}(h+4);6; max_{5 \le i \le c}(\dfrac{i(i-1)}{2}-h(2i-9))\},$$
	and thus we get 
	\begin{equation*}
	P_S \le XL_S -hn-h\sum\limits_{i\geq c+1}(2i-9)s_i.
	\end{equation*}
	For the medium index $i,$ we use the Crossing Lemma \ref{CL} and part (a) of Theorem \ref{ST} to imply that 
	\begin{equation*}
	\sum\limits_{j \ge i}(j-1)s_j \le \dfrac{\beta n^2}{2(i-1)^2},
	\end{equation*}
	thus we have 
	\begin{eqnarray*}
		&P_S&+P_M-XL_S \leq -hn-h\sum\limits_{i\geq c+1}(2i-9)s_i+ \sum\limits_{i=c+1}^{k-1}\left(\begin{array}{*{20}{c}}
			i\\2
		\end{array}\right)s_i\ \\
		&=&-hn-h\sum\limits_{i\geq k}(2i-9)s_i+ \sum\limits_{i=c+1}^{k-1}\bigg(\dfrac{i(i-1)}{2}-h(2i-9)\bigg)s_i\ \\
		&=&-hn-h\sum\limits_{i\geq k}(2i-9)s_i+\dfrac{1}{2}\left(\sum\limits_{i=c+1}^{k-1}(c-\dfrac{4hi-18h}{i-1})(i-1)s_i +\sum\limits_{i=c+1}^{k-1} \sum\limits_{j=i}^{k-1}(j-1)s_j\right)\\
		&=&-hn-h\sum\limits_{i\geq k}(2i-9)s_i+\dfrac{1}{2}\left( \sum\limits_{i=c+1}^{k-1}(c-4h +\dfrac{14h}{i-1})(i-1)s_i +\sum\limits_{i=c+1}^{k-1} \sum\limits_{j=i}^{k-1}(j-1)s_j\right)\\
		&\leq&-hn-h\sum\limits_{i\geq k}(2i-9)s_i+\dfrac{1}{2}\left( \sum\limits_{i=c+1}^{k-1}(c-4h+\dfrac{14h}{c})(i-1)s_i +\sum\limits_{i=c+1}^{k-1} \sum\limits_{j=i}^{k-1}(j-1)s_j\right)\\
		&=&-hn-h\sum\limits_{i\geq k}(2i-9)s_i+\dfrac{1}{2}\left((c-4h+\dfrac{14h}{c}) \sum\limits_{i=c+1}^{k-1}(i-1)s_i +\sum\limits_{i=c+1}^{k-1} \sum\limits_{j=i}^{k-1}(j-1)s_j\right)
		\end{eqnarray*}
	On the other hand, we have 
	\begin{align*}
	c-4h+\dfrac{14h}{c}&=c-4\dfrac{c^2-c-2}{4c-16}+\dfrac{7c^2-7c-14}{c(2c-8)}=\dfrac{c^2-3c-14}{2c(c-4)}\\
	&\Rightarrow c-4h+\dfrac{14h}{c} > 0 \  (\  \text{by } c \ge 29).
	\end{align*}
	So we get 
		
\begin{align*}
		P_S+P_M-XL_S&\le-hn-h\sum\limits_{i\geq k}(2i-9)s_i+\dfrac{1}{2}\left(\dfrac{c^2-3c-14}{2c(c-4)} \sum\limits_{i\geq c+1}(i-1)s_i +\sum\limits_{j=c+1}^{k-1} \sum\limits_{i\geq  j}(i-1)s_i\right)\\ &\leq-hn-h\sum\limits_{i\geq k}(2i-9)s_i+\left(\dfrac{c^2-3c-14}{2c^3(c-4)}+\sum\limits_{i\geq c}\dfrac{1}{i^2}\right)\dfrac{\beta n^2}{4}.
	\end{align*}
	Thus, we now have 
	\begin{align*}
	\left(\begin{array}{*{20}{c}}
	n\\2
	\end{array} \right) - XL_S&=P_S+P_M+P_L-XL_S \\
	& \leq -hn-h\sum\limits_{i\geq k}(2i-9)s_i+\left(\dfrac{c^2-3c-14}{2c^3(c-4)}+\sum\limits_{i\geq c}\dfrac{1}{i^2}\right)\dfrac{\beta n^2}{4}+ \sum\limits_{i=k}^{l}\left(\begin{array}{*{20}{c}}
	i\\2
	\end{array} \right)s_i\\
	&= -hn+\left(\dfrac{c^2-3c-14}{2c^3(c-4)}+\sum\limits_{i\geq c}\dfrac{1}{i^2}\right)\dfrac{\beta n^2}{4}+ \sum\limits_{i\geq k}\left(\dfrac{i(i-1)}{2} -2hi+9h\right)s_i\\
	&= -hn+\left(\dfrac{c^2-3c-14}{2c^3(c-4)}+\sum\limits_{i\geq c}\dfrac{1}{i^2}\right)\dfrac{\beta n^2}{4}+ \sum\limits_{i\geq k}\left(\dfrac{i}{2} -2h+\dfrac{7h}{i-1}\right)(i-1)s_i\\
	&\leq -hn+\left(\dfrac{c^2-3c-14}{2c^3(c-4)}+\sum\limits_{i\geq c}\dfrac{1}{i^2}\right)\dfrac{\beta n^2}{4}+ \dfrac{l}{2}|E(G_k)|\\
	&\le -hn+\left(\dfrac{c^2-3c-14}{2c^3(c-4)}+\sum\limits_{i\geq c}\dfrac{1}{i^2}\right)\dfrac{\beta n^2}{4}+ \dfrac{l}{2}\alpha n.
	\end{align*}
	So we get
	\begin{align*}
	L_S&\ge \bigg(\dfrac{1}{2}-\dfrac{\beta}{4}\left(\dfrac{c^2-3c-14}{2c^3(c-4)}+\sum\limits_{i\geq c}\dfrac{1}{i^2}\right)\bigg)\dfrac{n^2}{X}+\bigg(h-\dfrac{1}{2}-\dfrac{l\alpha}{2}\bigg)\dfrac{n}{X}\\
	&\ge \bigg(\dfrac{1}{2}-\dfrac{\beta}{4}\left(\dfrac{c^2-3c-14}{2c^3(c-4)}+\sum\limits_{i\geq c}\dfrac{1}{i^2}\right)\bigg)\dfrac{n^2}{X}-\dfrac{l\alpha n}{2X}.
	\end{align*}
	On the other hand, $X =h+1=\dfrac{c^2+3c-18}{4c-16},$ we thus get
	\begin{align*}
	L_S\ge \bigg(1-\dfrac{\beta}{2}\left(\dfrac{c^2-3c-14}{2c^3(c-4)}+\sum\limits_{i\geq c}\dfrac{1}{i^2}\right)\bigg)\dfrac{2c-8}{c^2+3c-18}n^2-\dfrac{(2c-8)\alpha}{c^2+3c-18}ln. 
	\end{align*}
	Theorem \ref{Theorem2} is proved.
\end{proof}
For the case $c=36,$ we get the following.
\begin{corollary}\label{t21}
	Every set $P$ of $n$ points with at most $l$ collinear determines at least $\dfrac{1}{39}n^2- \dfrac{1}{3}ln$ lines with at most 36 points.
\end{corollary}
We now apply Theorem \ref{Theorem2} to give the proof of Main theorem \ref{m2}.
\begin{proof}
	We may assume that $l$ is the size of the longest line. For some integer $c \geq 29$, then by Theorem \ref{Theorem2} we have $L \ge L_S \ge A(c)n^2-B(c)nl$ for some $A(c)$ and $B(c)$ evident in the theorem. Observe that,
	\begin{eqnarray*}
		A(c)&=&\bigg(1-\dfrac{\beta}{2}\left(\dfrac{c^2-3c-14}{2c^3(c-4)}+\sum\limits_{i\geq c}\dfrac{1}{i^2}\right)\bigg)\dfrac{2c-8}{c^2+3c-18} \\
		B(c)&=&\dfrac{(2c-8)\alpha}{c^2+3c-18}.
	\end{eqnarray*}
		We note that, 
	\begin{eqnarray*}
		\dfrac{2A}{1+2B} &\ge& \epsilon \\
		\Rightarrow A &\ge& \dfrac{\epsilon}{2}+ B \epsilon -\dfrac{\epsilon^2}{2}\\
		\Rightarrow An &\ge& \dfrac{\epsilon n}{2} +(B-\dfrac{\epsilon}{2}) \epsilon n\\
		\Rightarrow An &\ge& \dfrac{\epsilon n}{2} + (B-\dfrac{\epsilon}{2})l\\
		\Rightarrow An^2- Bnl &\ge& \dfrac{\epsilon n (n-l)}{2}.
	\end{eqnarray*}
	So we can find the maximum of  $\dfrac{2A(c)}{1+2B(c)}$ to get a largest number $\epsilon.$ Now, set $c=44$ we get $\epsilon \le \dfrac{1}{30.2}.$ So we choose $\epsilon = \dfrac{1}{30.5}$ to complete Main theorem \ref{m2}.
\end{proof}
We now get Main theorem \ref{m3} by  using Main theorem \ref{m2} and the following observation:
\begin{theorem}(\cite{PW}) \label{T5}
	Let $P$ be a set of $n$ non-collinear points in a plane. Then at least half the lines determined by
	$P$ contain at most 3 points.
\end{theorem}

\end{document}